\documentclass[12pt]{amsart}
\usepackage{amssymb,amsmath,amsthm,latexsym}
\usepackage{mathrsfs}
\usepackage{a4wide}
\usepackage{eucal}
\usepackage[usenames,dvipsnames]{color}
\usepackage[utf8]{inputenc}
\usepackage[T1]{fontenc}
\usepackage{dsfont}

\newtheorem{theorem}{Theorem}[section]
\newtheorem{lemma}[theorem]{Lemma}

\theoremstyle{remark}

\newtheorem{remark}[theorem]{Remark}
\newtheorem*{remark*}{Remark}
\theoremstyle{definition}

\numberwithin{equation}{section}
\makeatother

\newcommand{\vertiii}[1]{{\left\vert\kern-0.25ex\left\vert\kern-0.25ex\left\vert #1 
    \right\vert\kern-0.25ex\right\vert\kern-0.25ex\right\vert}}

\newcounter{smallromans}

\newenvironment{romanenumerate}
{\begin{list}{{\normalfont\textrm{(\roman{smallromans})}}}%
  {\usecounter{smallromans}\setlength{\itemindent}{0cm}%
   \setlength{\leftmargin}{5.5ex}\setlength{\labelwidth}{5.5ex}%
   \setlength{\topsep}{.5ex}\setlength{\partopsep}{.5ex}%
   \setlength{\itemsep}{0.1ex}}}%
{\end{list}}

\newcounter{smallromansdash}

{\end{list}}

\newcounter{bigromans} 
  {\end{list}}

\title[Embedding Banach spaces into $\ell_\infty^c(\Gamma)$]{Embedding Banach spaces into the space of bounded functions with countable support}
\author[W.B.~Johnson]{William B. Johnson}
\author[T.~Kania]{Tomasz Kania}
\address{Department of Mathematics, Texas A\&M University, College Station, TX 77843, U.S.A}
\email{johnson@math.tamu.edu}
\address{Institute of Mathematics, Czech Academy of Sciences, \v{Z}itn\'{a} 25, 115~67 Prague 1, Czech Republic}
\email{tomasz.marcin.kania@gmail.com}

\date{\today}
\thanks{The first-named author was supported in part by NSF DMS-1565826. The second-named author acknowledges with thanks funding from funding received from GA\v{C}R project 17-27844S; RVO 67985840 (Czech Republic)}

\begin{document}
\begin{abstract}We prove that a WLD subspace of the space $\ell_\infty^c(\Gamma)$ consisting of all bounded, countably supported functions on a  set $\Gamma$ embeds isomorphically  into $\ell_\infty$ if and only if it does not contain isometric copies of $c_0(\omega_1)$. Moreover, a subspace of $\ell_\infty^c(\omega_1)$ is constructed that has an unconditional basis, does not embed into $\ell_\infty$, and whose every weakly compact subset is separable (in particular, it cannot contain any isomorphic copies of $c_0(\omega_1)$). \end{abstract}
\maketitle

\section{Introduction}
It is classical that every separable Banach space is isometrically isomorphic to a subspace of  $\ell_\infty$, the space of bounded sequences with the supremum norm. Since every weakly compact subset of $\ell_\infty$ is separable,  any weakly compactly generated space; in particular, any reflexive space; that admits an injective bounded linear operator into  $\ell_\infty$  must be separable. (A Banach space is \emph{weakly compactly generated}, WCG for short, when it contains a  weakly compact subset whose linear span is dense.) For this reason, there is no bounded linear injection from $c_0(\Gamma)$ into $\ell_\infty$ when the set $\Gamma$ is uncountable. Nevertheless, $c_0(\omega_1)$ sits naturally as a subspace of $\ell_\infty$'s close cousin, the space $\ell_\infty^c(\Gamma)$, which consists of all bounded scalar-valued functions on $\Gamma$ that are non-zero on at most countably many points in $\Gamma$. \smallskip

The aim of this note is to study Banach spaces that embed into $\ell_\infty^c(\omega_1)$ but do not embed into $\ell_\infty$ and their relation to containment of isomorphic, or even isometric, copies of $c_0(\omega_1)$. In particular, we prove that a non-separable weakly Lindel\"of determined (WLD) subspace of $\ell_\infty^c(\Gamma)$ contains an isometric copy of $c_0(\omega_1)$. (A Banach space $X$ is WLD provided for some set $\Gamma$ there exists an injective linear operator $T\colon X^* \to \ell_\infty^c(\Gamma)$ that is continuous as a map from $X^*$ with the weak* topology to $\ell_\infty^c(\Gamma)$ with the topology of pointwise convergence.) \smallskip 

The notation is standard.  We just mention that all operators are assumed to be bounded and linear, and an  isomorphism is a bounded linear operator that is bounded below on the unit sphere of its domain. We consider cardinal numbers as initial ordinal numbers. 
%The successor of a cardinal number $\lambda$ is denoted by $\lambda^+$. 
A~cardinal number $\lambda$ is \emph{regular} whenever a set of cardinality $\lambda$ cannot be expressed as a~union of fewer than $\lambda$ sets that have cardinality less than $\lambda$.

\section{The results}
For a WLD space $X$, the density character of $X^*$ endowed with the weak* topology is the same as of $X$ with the norm topology (\cite[Proposition 5.40]{czechbook}), hence the existence of a~bounded linear injection $T\colon X\to \ell_\infty$ implies that $X$ is separable. Indeed, by Goldstine's theorem, $\ell_\infty^* = \ell_1^{**}$ is weak*-separable. By injectivity of $T$, the adjoint map $T^*\colon \ell_\infty^*\to X^*$ has dense range, and so $X^*$ has a weak*-separable dense subspace, therefore it must be weak*-separable. As $X$ is WLD, also $X$ must be separable. We shall employ this fact to construct copies of $c_0(\omega_1)$ in non-separable WLD subspaces of $\ell_\infty^c(\Gamma)$. (In the case of $X=c_0(\Gamma)$ the result was already recorded in \cite{r-s} and \cite[Lemma 6]{hn}).
\begin{theorem}\label{main}Let $\Gamma$ be a set and let $X$ be a WLD subspace of $\ell_\infty^c(\Gamma)$. Then the following assertions are equivalent:
\begin{romanenumerate}
\item\label{t1} $X$ is separable,
\item\label{t2} $X$ embeds into $\ell_\infty$,
\item\label{t2a}there exists a bounded linear injection from $X$ into $\ell_\infty$,
\item\label{t3} $X$ does not contain a subspace that is isomorphic to  $c_0(\omega_1)$,
\item\label{t4} $X$ does not contain a subspace that is isometrically  isomorphic to  $c_0(\omega_1)$.
\end{romanenumerate}
In particular, every reflexive subspace of $\ell_\infty^c(\Gamma)$ is separable.
 \end{theorem}
First we introduce some notation. Let $\Gamma$ be a set and let $\Lambda\subseteq \Gamma$. Consider the contractive projection $P_\Lambda\colon \ell_\infty^c(\Gamma)\to \ell_\infty^c(\Gamma)$ given by
$$(P_\Lambda f)(\gamma) =\left\{\begin{array}{ll}f(\gamma)\, &\gamma \in \Gamma,\\ 0, &\gamma \in \Lambda\setminus \Gamma \end{array}   \right. \qquad \big(f\in \ell_\infty^c(\Gamma)\big).$$
We  identify   the range of $P_\Gamma$ with the space $\ell_\infty^c(\Gamma)$. Certainly, when $\Gamma$ is countably infinite, the range of $P_\Gamma$ is isometrically isomorphic to $\ell_\infty$.

\begin{proof}The implications \eqref{t1} $\Rightarrow$ \eqref{t2} $\Rightarrow$ \eqref{t2a} $\Rightarrow$ \eqref{t3} $\Rightarrow$ \eqref{t4} are clear. We have already observed in the introduction that for WLD spaces \eqref{t2a} $\Rightarrow$ \eqref{t1}. We shall prove now that \eqref{t4} $\Rightarrow$ \eqref{t2a} by contraposition.\smallskip 

Suppose that there is no bounded linear injection from $X$ into $\ell_\infty$. In particular, for every countable set $\Lambda\subset \Gamma$, the restriction operator $P_\Gamma|_X$ is not injective, which means that $P_\Lambda f_\Lambda = 0$ for some unit vector $f_\Lambda$ in the range of $P_\Lambda$. Consequently, it is possible to choose by transfinite recursion an uncountable family of pairwise disjoint countable subsets $(\Lambda_\alpha)_{\alpha<\omega_1}$ and unit vectors $f_\alpha\in \ell_\infty(\Lambda_\alpha)\cap X$. Then, the closed linear span of $\{f_\alpha\colon \alpha < \omega_1\}$ is isometric to $c_0(\omega_1)$.
\end{proof}
\begin{remark}Theorem A implies that the unit sphere of a non-separable WLD subspace of $\ell_\infty^c(\Gamma)$ contains an uncountable symmetrically $(1+)$-separated subset, that is, a set $A$ such that $\|x\pm y\| > 1$ for distinct $x,y\in A$; this is because $c_0(\omega_1)$ has this property. This observation complements \cite[Corollary 3.6]{hkr}, where it was proved that WLD spaces of density greater than the continuum contain such sets. It should be noted however that not every renorming of $c_0(\omega_1)$ embeds isometrically into $\ell_\infty^c(\Gamma)$, as at least under the Continuum Hypothesis, there exists a renorming of   $c_0(\omega_1)$ that does not contain isometric copies of itself (\cite[Theorem 5.9]{hkr}).\end{remark}

The hypothesis of being WLD cannot be removed completely from the statement of Theorem A. Before we give a relevant example, we prove a simple lemma.
\begin{lemma}\label{small}Let $X$ be a subspace of $\ell_\infty^c(\omega_1)$. If $X$ embeds into $\ell_\infty$, then there is a $\alpha<\omega_1$ such that the operator $P_{[0,\alpha)}|_X$ is bounded below; that is, bounded below on the unit sphere of $X$.\end{lemma}
\begin{proof} Since  $\ell_\infty$ is injective, there is an operator $J\colon \ell_\infty^c(\omega_1) \to \ell_\infty$ so that the restriction of $J$ to $X$ is bounded below (indeed, $J$ is any extension of an embedding of $X$ into $\ell_\infty$ to $\ell_\infty^c(\omega_1)$).  It is therefore enough to observe that for any operator $J\colon \ell_\infty^c(\omega_1) \to \ell_\infty$ there is a countable ordinal $\alpha$ such that $J$ vanishes on $\ell_\infty^c\big([\alpha,\omega_1)\big)$ (because then $J$ factors through the quotient $\ell_\infty^c(\omega_1) / \ell_\infty^c\big([\alpha,\omega_1)\big),$
which is isomorphic to $\ell_\infty^c\big([0,\alpha)\big)\cong \ell_\infty)$. \end{proof}
\begin{theorem} \label{example} There exists a subspace $Z$ of $\ell_\infty^c(\omega_1)$ with an unconditional basis such that $Z$ does not embed into $\ell_\infty$   and   $c_0(\omega_1)$ does not embed into $Z$. Moreover, $Z$ contains an isomorphic copy of $\ell_1(\omega_1)$ and there is an injective operator from $Z$ into $\ell_\infty$.\end{theorem}
\begin{proof}Since $\ell_\infty$ contains an isometric copy of $\ell_1(\mathfrak{c})$, we may fix a countable set $\Gamma_0$ in $\omega_1$ and a family of unit vectors $(f_\alpha)_{\alpha<\omega_1}$ in the range of $P_{\Gamma_0}$ that is isometrically equivalent to the unit vector basis of $\ell_1(\omega_1)$. Let $\alpha = \gamma(\alpha) + n(\alpha)$ be a countable ordinal, where $\gamma(\alpha)$ is a (possibly zero) limit ordinal and $n(\alpha)$ is a finite ordinal. We set 
$$z_\alpha = e_\alpha + \tfrac{1}{n(\alpha)+1} z_\alpha \qquad (\alpha<\omega_1),$$
where $(e_\alpha)_{\alpha < \omega_1}$ is the standard unit vector basis of $c_0(\omega_1) \subset \ell_\infty^c(\omega_1)$. Let $Z$ be the closed linear span of $z_\alpha$ ($\alpha<\omega_1$) in $\ell_\infty^c(\omega_1)$. Then $(z_\alpha)_{\alpha<\omega_1}$ is a 1-unconditional basis for $Z$. Moreover the operator $P_{\Gamma_0}|_Z$ is injective, so $Z$ cannot contain non-separable weakly compact sets as every weakly compact subset of $\ell_\infty$ is separable. In particular, $X$ does not contain any isomorphic copies of $c_0(\omega_1)$.\smallskip

By Lemma~\ref{small}, $Z$ does not embed into $\ell_\infty$ because the operator $P_{[0,\alpha)}|_X$ is not bounded below  for any countable ordinal $\alpha$.\smallskip

Finally, we remark that $P_{\Gamma_0}$ is an isomorphism when restricted to the copy of $\ell_1(\omega_1)$ spanned by the family $\{z_\alpha\colon \alpha<\omega_1, n(\alpha)=0\}.$\end{proof}

\begin{remark} As was noted in the proof of   Theorem \ref{example}, the example $Z$ does not embed into $\ell_\infty$ but there is an injective operator from $Z$ into $\ell_\infty$. The first space having these properties was constructed in \cite{jl}, but that space does not have an unconditional basis.
\end{remark}

\subsection{An extension to higher densities}
For every cardinal $\lambda$, there is a natural generalisation of the space $\ell_\infty^c(\Gamma)$; namely, $\ell_\infty^\lambda(\Gamma)$, the subspace of $\ell_\infty(\Gamma)$ that comprises functions whose supports have cardinality strictly less than $\lambda$. In this notation, $\ell_\infty^c(\Gamma) = \ell_\infty^{\omega_1}(\Gamma)$. We note that Theorem~\ref{main} has a natural counterpart for spaces $\ell_\infty^\lambda(\Gamma)$, whenever $\lambda$ is a regular cardinal.

\begin{theorem}\label{higher}Let $\Gamma$ be a set, $\lambda$ a regular cardinal number, and let $X$ be a subspace of $\ell_\infty^\lambda(\Gamma)$. Then the following assertions are equivalent:
\begin{romanenumerate}
\item\label{ts1} $w^*{\rm -dens}\, X^* < \lambda$,
\item\label{ts2} $X$ embeds into $\ell_\infty(\kappa)$ for some $\kappa < \lambda$,
\item\label{ts2a}there exists a bounded linear injection from $X$ into $\ell_\infty(\kappa)$ for some $\kappa < \lambda$,
\item\label{ts3} $X$ does not contain a subspace that is isomorphic to  $c_0(\lambda)$,
\item\label{ts4} $X$ does not contain a subspace that is isometrically  isomorphic to  $c_0(\lambda)$.
\end{romanenumerate}
\begin{proof}Note that \eqref{ts2a} $\Rightarrow$ \eqref{ts1}. Indeed, if there is a bounded linear injection $T$ from $X$ into $\ell_\infty(\kappa)$, then $T^*$ has weak*-dense range. As (by Goldstine's theorem) the weak* density of $\ell_\infty(\kappa)^*$ is $\kappa$, the conclusion follows. The implication \eqref{ts2a} $\Rightarrow$ \eqref{ts3} follows from the fact that the weak* density of $c_0(\lambda)^*$ is $\lambda$ and thus there is no bounded linear injection from $c_0(\lambda)$ to $\ell_\infty(\kappa)$ for $\kappa < \lambda$ (see, \emph{e.g.}, \cite[Fact 4.10]{czechbook}). (We remark in passing that the implication \eqref{ts2} $\Rightarrow$ \eqref{ts3} was proved directly in \cite[Proposition 3.4]{jks}.) As previously, it is enough to prove that \eqref{ts4} $\Rightarrow$ \eqref{ts2a}. \smallskip 

Assume contrapositively that for all $\kappa < \lambda$ there is no bounded linear injection from $X$ into $\ell_\infty(\kappa)$. In particular, $|\Gamma| \geqslant \lambda$ as otherwise $\ell_\infty^\lambda(\Gamma) = \ell_\infty(\Gamma)$ but $X$ is a subspace of $\ell_\infty^\lambda(\Gamma)$.\smallskip  
%Without loss of generality we may assume that $|\Gamma| = \lambda$. 

Let $\mathcal A$ be a family of non-zero vectors in $X$ that is maximal with respect to the property that the vectors have  pairwise disjoint supports. If $\mathcal{A}$ has cardinality $\lambda$, the conclusion follows as $\mathcal{A}$ spans an isometric copy of $c_0(\lambda)$. So assume that $|\mathcal{A}| < \lambda$. Let $$ \Lambda = \bigcup_{f\in \mathcal A} {\rm supp}\, f. $$
As $\lambda$ is regular (and $|\Gamma| \geqslant \lambda$), $|\Lambda | < \lambda$. Consequently, by maximality of  $\mathcal{A}$,  the contractive projection $P_\Lambda : \ell_\infty^\lambda(\Gamma) \to \ell_\infty^\lambda(\Lambda)$  maps $X$ injectively into  $ \ell_\infty(\Lambda)$; a contradiction
\end{proof}
 \end{theorem}

The conclusion of Theorem~\ref{higher} fails for singular cardinal numbers. Indeed, let us take $\lambda = \omega_\omega = \lim_{n\to \infty} \omega_n.$ The space  $\ell_\infty( \omega_n)$ contains an isometric copy of the Hilbert space $\ell_2(\omega_n)$. In particular, the $c_0$-direct sum of $\ell_2(\omega_n)$ ($n\in \mathbb N$) embeds isometrically into $\ell_\infty^\lambda(\lambda)$, has density $\lambda$, and is WCG (and even Asplund). On the other hand, it does not contain $c_0(\lambda)$.


\begin{thebibliography}{99}
\bibitem{am} S. Argyros and S. Mercourakis, On Weakly Lindel\"of Banach Spaces, \emph{Rocky Mountain J. Math.}~\textbf{23} (1993), 395--446.
%\bibitem{gmm} A.~S.~Granero, M.~M.~Jim\'enez, and J.~P.~Moreno, On the nonseparable subspaces of $J(\eta)$ and $C([1,\eta])$, \emph{Math.~Nachr.} \textbf{221} (2001), 75--85.
\bibitem{hkr}P.~H\'ajek, T.~Kania, and T.~Russo, Separated sets and Auerbach systems in Banach spaces, preprint (2018).
\bibitem{czechbook} P.~H\'ajek, V.~Montesinos, J.~Vanderverff, and V.~Zizler, \emph{Biorthogonal systems in Banach spaces}, CMS Books in Mathematics/Ouvrages de Mathematiques de la SMC, 26. Springer-Verlag, New York, 2008.
\bibitem{hn}P.~H\'ajek and M.~Novotn{{\'y}}, Distortion of Lipchitz functions on $c_0(\Gamma)$, \emph{Proc.~Amer.~Math.~Soc.}, \textbf{146} (2018), 2173--2180.
\bibitem{jks} W.~B.~Johnson, T.~Kania, and G.~Schechtman,  Closed ideals of operators on and complemented subspaces of Banach spaces of functions with countable support, \emph{Proc.~Amer.~Math.~Soc.} \textbf{144} (2016), no. 10, 4471--4485.
\bibitem{jl} W.~B. Johnson and J. Lindenstrauss, Some remarks on weakly compactly generated Banach spaces, \emph{Israel J. Math.} \textbf{17} (1974), 219--230. Correction ibid. \textbf{32} (1979), 382--383.
\bibitem{r-s} B.~Rodr{{\'i}}guez-Salinas, On the complemented subspaces of $c_0(I)$ and $\ell_p(I)$ for $1<p<\infty$, \emph{Atti Sem. Mat. Fis. Univ. Modena} \textbf{42} (1994), no. 2, 399--402.
\end{thebibliography}
\end{document}